\documentclass[12pt]{article}

\usepackage{amsmath}
\usepackage{amsfonts}
\usepackage{amssymb}
\usepackage{amsthm}
\usepackage{ifthen}
\usepackage{tikz}
\usetikzlibrary{shapes,arrows}

\newtheorem{theorem}{Theorem}[section]
\newtheorem{lemma}[theorem]{Lemma}
\newtheorem{proposition}[theorem]{Proposition}
\newtheorem{corollary}[theorem]{Corollary}

\newcommand{\pl}{\phantom{-}}
\newcommand{\mi}{-}

\title{Permutations with Ascending and Descending Blocks}
\author{Jacob Steinhardt}

\begin{document}

\maketitle

\begin{abstract}
We investigate permutations in terms of their cycle structure and descent set. To do this, we generalize the classical bijection of Gessel and Reutenauer to deal with permutations that have some ascending and some descending blocks. We then provide the first bijective proofs of some known results. We also solve some problems posed in \cite{EFW} by Eriksen, Freij, and W\"{a}stlund, who study derangements that descend in blocks of prescribed lengths.
\end{abstract}

\section{Introduction} 


We consider permutations in terms of their descent set and conjugacy class (equivalently, cycle structure). Let $\pi$ be a permutation on $\{1,\ldots,n\}$. An \emph{ascent} of $\pi$ is an index $i$, $1 \leq i < n$, such that $\pi(i) < \pi(i+1)$. A \emph{descent} of $\pi$ is such an index with $\pi(i) > \pi(i+1)$.


The study of permutations by descent set and cycle structure goes back at least as far as 1993, when Gessel and Reutenauer enumerated them using symmetric functions \cite{GR}. In their proof, they obtained a bijection from permutations with at most a given descent set to multisets of necklaces with certain properties. By a \emph{necklace} we mean a directed cycle, taken up to cyclic rotations, where the vertices are usually assigned colors or numbers. Multisets of necklaces are usually referred to as \emph{ornaments}. Figure \ref{necklace-example} illustrates these terms. 


\begin{figure}[b]

\centering

\begin{tikzpicture}[shorten >=1pt,->]
	\tikzstyle{vw}=[circle,draw=black!100,fill=white,minimum size=17pt,inner sep=0pt]
	\tikzstyle{vb}=[circle,draw=black!100,fill=black!25,minimum size=17pt,inner sep=0pt]


\node[xshift=0cm,yshift=0.2cm] {(a)};

	\foreach \name/\angle/\color in {P1-1/162/vw,
				P1-2/234/vb, P1-3/306/vb,	P1-4/18/vw, P1-5/90/vb}
		\node[\color,xshift=1.7cm,yshift=0.2cm] (\name) at (\angle:1cm) {};

\node[xshift=3.8cm,yshift=0.2cm] {(b)};

	\foreach \name/\angle/\color in {P2-1/162/vb,
				P2-2/234/vb, P2-3/306/vw,	P2-4/18/vb, P2-5/90/vw}
		\node[\color,xshift=5.5cm,yshift=0.2cm] (\name) at (\angle:1cm) {};

	\foreach \from/\to in {1/2,2/3,3/4,4/5,5/1}{
		\draw (P1-\from) -- (P1-\to);
		\draw (P2-\from) -- (P2-\to);
	}

\node[xshift=7.6cm,yshift=0.2cm] {(c)};

	\foreach \name/\angle/\color in {T1-1/90/vw,T1-2/210/vw,T1-3/330/vb}
		\node[\color,xshift=9.3cm,yshift=0.0cm] (\name) at (\angle:1cm) {};
	
	\foreach \name/\angle/\color in {T2-1/90/vb,T2-2/210/vw,T2-3/330/vb}
		\node[\color,xshift=12.0cm,yshift=0.0cm] (\name) at (\angle:1cm) {};

	\foreach \from/\to in {1/2,2/3,3/1}{
		\draw (T1-\from) -- (T1-\to);
		\draw (T2-\from) -- (T2-\to);
	}

\node[vb,xshift=14.4cm,yshift=0.2cm] (X) at (0,0) {};

\draw [->] (X) to [loop above] (X);

\end{tikzpicture}

\caption{Examples of necklaces and ornaments. (a) and (b) are two different representations of the same necklace with 5 vertices. (c) is an ornament with two different 3-cycles and a 1-cycle.}
\label{necklace-example}

\end{figure}
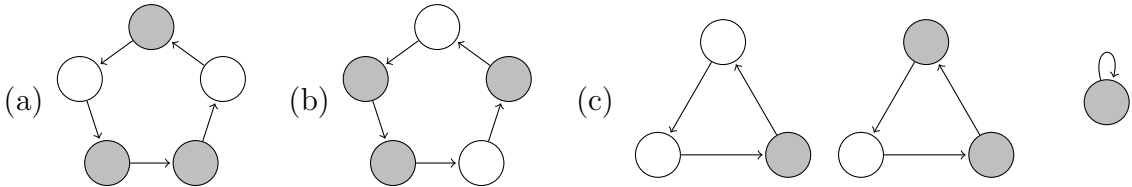

The Gessel-Reutenauer bijection preserves cycle structure. It also forgets other structure that is not so relevant, making it easier to study permutations by cycle structure and descent set. We will restate Gessel's and Reutenauer's result to bring it closer to the language of more recent work (\cite{HX}, \cite{EFW}). Choose $a_1,\ldots,a_k$ with $a_1+\cdots+a_k = n$, and partition $\{1,\ldots,n\}$ into consecutive blocks $A_1,\ldots,A_k$ with $|A_i| = a_i$. An \emph{$(a_1,\ldots,a_k)$-ascending permutation} is a permutation $\pi$ that ascends within each of the blocks $A_1,\ldots,A_k$. This is the same as saying that the descent set of $\pi$ is contained in $\{a_1,a_1+a_2,\ldots,a_1+a_2+\cdots+a_{k-1}\}$. In this language, the Gessel-Reutenauer bijection is a map from $(a_1,\ldots,a_k)$-ascending permutations to ornaments that preserves cycle structure.

We provide a generalization of the Gessel-Reutenauer bijection to deal with both ascending and descending blocks. Let $A = (a_1,\ldots,a_k)$ and $S \subset \{1,\ldots,k\}$. Then an \emph{$(a_1,\ldots,a_k,S)$-permutation} (or just an $(A,S)$-permutation if $a_1,\ldots,a_k$ are clear from context) is a permutation that descends in the blocks $A_i$ for $i \in S$ and ascends in all of the other blocks. We generalize the Gessel-Reutenauer bijection to give a cycle-structure-preserving bijection from the $(A,S)$-permutations to ornaments with certain properties. Our bijection can be thought of as equivalent to Reiner's \cite{Rei1} bijection for signed permutations, as a descent for normal permutations is the same as an ascent over negative values for signed permutations.

Both here and in \cite{GR}, the Gessel-Reutenauer bijection is easy to describe. We take a permutation $\pi$, write it as a product of disjoint cycles, and replace each element of each cycle by the block it belongs to. A permutation and its image under the bijection is illustrated later in the paper, in Figures \ref{GR-bijection-example-2} and \ref{GR-bijection-example}, respectively. Since the Gessel-Reutenauer bijection forgets so much structure, the surprising thing is that it is injective.

We describe the image of our bijection in Theorem \ref{GR-bijection}. Using this bijection, we obtain a second bijection onto ornaments, but this time the ornaments have properties that are easier to describe. The tradeoff is that the bijection no longer preserves cycle structure, but it is not too difficult to describe how the cycle structure changes. This bijection is described in Proposition \ref{auxiliary-bijection}.



These bijections allow us to take a purely combinatorial approach to the problems considered in \cite{HX} and \cite{EFW}. In \cite{HX}, Han and Xin, motivated by a problem of Stanley \cite{Sta1}, study the \emph{$(a_1,\ldots,a_k)$-descending} derangements, meaning derangements that descend in each of the blocks $A_1,\ldots,A_k$ (so, in our language, the case when $S = \{1,\ldots,k\}$). Han and Xin use symmetric functions to prove their results. In \cite{EFW}, Eriksen, Freij, and W\"{a}stlund also study the $(a_1,\ldots,a_k)$-descending derangements, but they use generating functions instead of symmetric functions.

Eriksen et al.\ show that the number of $(a_1,\ldots,a_k)$-descending derangements is symmetric in $a_1,\ldots,a_k$ and ask for a bijective proof of this fact. We obtain a bijective proof of the following stronger statement.

\newtheorem*{hack1}{Corollary \ref{EFW-symmetric}\footnote{Sergei Elizalde proves a slightly less general version of Corollary \ref{EFW-symmetric} as Proposition 4.2 of \cite{Eli}.}}
\begin{hack1}
Let $\sigma$ be a permutation of $\{1,\ldots,k\}$ and let $\mathcal{C}$ be a conjugacy class in $S_n$. The number of $(a_1,\ldots,a_k,S)$-permutations in $\mathcal{C}$ is the same as the number of $(a_{\sigma(1)},\ldots,a_{\sigma(k)},\sigma(S))$-permutations in $\mathcal{C}$.
\end{hack1}

\noindent Eriksen et al.\ also show that the number of $(a_1,\ldots,a_k)$-descending derangements is

\[
\sum_{0 \leq b_m \leq a_m, m = 1,\ldots,k} (-1)^{\sum b_i} \binom{\sum (a_i-b_i)}{a_1-b_1,\ldots,a_k-b_k}.
\]

\noindent They do this using the generating function

\[
\frac{1}{1-x_1-\cdots-x_k}\left(\frac{1}{1+x_1}\cdots\frac{1}{1+x_k}\right)
\]

\noindent for the $(a_1,\ldots,a_k)$-descending derangements, which first appears in \cite{HX}. They ask for a combinatorial proof of their formula using inclusion-exclusion. They also ask for a similar enumeration of the $(a_1,\ldots,a_k)$-ascending derangements. We provide both of these as a corollary to Proposition \ref{auxiliary-bijection}.

\newtheorem*{hack1.5}{Corollary \ref{EFW-enumerate}}
\begin{hack1.5}
The number of $(A,S)$-derangements is the coefficient of $x_1^{a_1}\cdots x_k^{a_k}$ in 

\[
\frac{1}{1-x_1-\cdots-x_k} \left(\frac{\prod_{i \not\in S} (1-x_i)}{\prod_{i \in S} (1+x_i)}\right).
\]

\noindent Let $l_m = a_m$ if $m \in S$ and let $l_m = 1$ otherwise. The number of $(A,S)$-derangements is also

\[
\sum_{0 \leq b_m \leq l_m, m=1,\ldots,k} (-1)^{\sum b_i} \binom{\sum (a_i-b_i)}{a_1-b_1,\ldots,a_k-b_k}.
\]
\end{hack1.5}

It is also possible to prove Corollary \ref{EFW-enumerate} more directly using some structural lemmas about $(A,S)$-derangements and standard techniques in recursive enumeration. We do this in \cite{Ste2}.


Our bijective methods also apply to some of the results in the original paper by Gessel and Reutenauer. The $(a_1,\ldots,a_k)$-ascending permutations are all permutations with at most a given descent set. By using inclusion-exclusion on the $(a_1,\ldots,a_k)$-ascending permutations, we can study the number of permutations with \emph{exactly} a given descent set. We can do the same thing with the $(a_1,\ldots,a_k)$-descending permutations. It turns out that comparing the two allows us to see what happens when we take the complement of the descent set. In \cite{GR}, Gessel and Reutenauer prove the following two theorems.

\newtheorem*{hack1.6}{Theorem 4.1 of \cite{GR}}
\begin{hack1.6}
\label{GR-thm-4.1}
Associate to each conjugacy class of $S_n$ a partition $\lambda$ based on cycle structure. If $\lambda$ has no parts congruent to $2$ modulo $4$ and every odd part of $\lambda$ occurs only once, then the number of permutations of cycle structure $\lambda$ with a given descent set is equal to the number of permutations of cycle structure $\lambda$ with the complementary descent set.
\end{hack1.6}

\newtheorem*{hack1.7}{Theorem 4.2 of \cite{GR}}
\begin{hack1.7}
\label{GR-thm-4.2}
The number of involutions in $S_n$ with a given descent set is equal to the number of involutions in $S_n$ with the complementary descent set.
\end{hack1.7}

We obtain Theorem 4.1 of \cite{GR} as a consequence of Corollary \ref{GR-4.1} by setting $S$ to $\emptyset$. Corollary \ref{GR-4.1} deals with permutations with \emph{at least} a given ascent or descent set, but as noted before we can apply inclusion-exclusion to get the same result about pemutations with \emph{exactly} a given ascent or descent set.

\newtheorem*{hack2}{Corollary \ref{GR-4.1}}
\begin{hack2}
Associate to each conjugacy class $\mathcal{C}$ of $S_n$ a partition $\lambda$ of $n$ based on cycle structure.

The number of $(A,S)$-permutations in $\mathcal{C}$ is the same if we replace $S$ by $\{1,\ldots,k\}\backslash S$, assuming that all odd parts of $\lambda$ are distinct and $\lambda$ has no parts congruent to $2$ mod $4$.
\end{hack2}

To our knowledge, this is the first bijective proof of Theorem 4.1 of \cite{GR}. We also obtain the following generalization of Theorem 4.2 of \cite{GR}.

\newtheorem*{hack3}{Corollary \ref{GR-4.2}}
\begin{hack3}
The number of $(A,S)$-involutions is the same if we replace $S$ by $\{1,\ldots,k\}\backslash S$.
\end{hack3}

This is the first known bijective proof of Theorem 4.2 of \cite{GR}.


The rest of the paper is divided into five sections. In Section \ref{bijections}, we describe the two bijections used in the remainder of the paper and prove that they are bijections. In Section \ref{corollaries}, we prove the corollaries resulting from these bijections. In Section \ref{conclusion}, we discuss directions of further research, including the study of some maps related to the Gessel-Reutenauer bijection as well as a generalization of a polynomial identity arising in \cite{EFW}.

We also define all the terms used in this paper in Section \ref{defns}, which occurs after the Acknowledgements and before the Bibliography. These terms are all defined either in the introduction or as they appear in the paper, but we have also collected them in a single location for easy reference.

\section{The Two Bijections}
\label{bijections}

We now describe our two bijections. Here and later, we will have occasion to talk about ornaments labeled by $\{1,\ldots,k\}$. In this case we call the integers $1$ through $k$ \emph{colors}, the elements in $S$ \emph{descending colors}, and the elements not in $S$ \emph{ascending colors}.

Our first bijection is from the $(A,S)$-permutations with conjugacy class $\mathcal{C}$ to ornaments with the same cycle structure.


Before formally defining the bijection, we will give an illustrative example. Let us suppose that we were considering the $((8,10),\{1\})$-permutations---in other words, permutations that descend in a block of length $8$ and then ascend in a block of length $10$. In particular, we will take the permutation $\pi = 18 \ 17 \ 15 \ 14 \ 13 \ 12 \ 11 \ 9 \ 1 \ 2 \ 3 \ 4 \ 5 \ 6 \ 7 \ 8 \ 10 \ 16$. This permutation has cycle structure $(1 \ 18 \ 16 \ 8 \ 9)(2 \ 17 \ 10)(3 \ 15 \ 7 \ 11)(4 \ 14 \ 6 \ 12)(5 \ 13)$. If we replace each vertex in each cycle by the block it belongs to ($A_1$ or $A_2$), then we get $(1 \ 2 \ 2 \ 1 \ 2)(1 \ 2 \ 2)(1 \ 2 \ 1 \ 2)(1 \ 2 \ 1 \ 2)(1 \ 2)$, which corresponds to the ornament depicted in Figure \ref{GR-bijection-example}. Under our bijection, we send $\pi$ to this ornament.


\begin{figure}[b]

\centering

\begin{tikzpicture}[shorten >=1pt,->]
	\tikzstyle{vw}=[circle,draw=black!100,fill=white,minimum size=17pt,inner sep=0pt]
	\tikzstyle{vb}=[circle,draw=black!100,fill=black!25,minimum size=17pt,inner sep=0pt]

	\foreach \name/\angle/\color/\text in {P-1/162/vw/A,
				P-2/234/vb/B, P-3/306/vb/C,	P-4/18/vw/D, P-5/90/vb/E}
		\node[\color,xshift=0cm,yshift=0.2cm] (\name) at (\angle:1cm) {$\text$};

	\foreach \from/\to in {1/2,2/3,3/4,4/5,5/1}
		\draw (P-\from) -- (P-\to);
	
	\foreach \name/\angle/\color/\text in {T-1/90/vw/F,
				T-2/210/vb/G, T-3/330/vb/H}
		\node[\color,xshift=2.7cm,yshift=0cm] (\name) at (\angle:1cm) {$\text$};

	\foreach \from/\to in {1/2,2/3,3/1}
		\draw (T-\from) -- (T-\to);


	\foreach \name/\angle/\color/\text in {S1-1/135/vw/I,
				S1-2/225/vb/J, S1-3/315/vw/K, S1-4/45/vb/L}
		\node[\color,xshift=5.3cm,yshift=0.1cm] (\name) at (\angle:1cm) {$\text$};

	\foreach \name/\angle/\color/\text in {S2-1/135/vw/M,
				S2-2/225/vb/N, S2-3/315/vw/O, S2-4/45/vb/P}
		\node[\color,xshift=7.8cm,yshift=0.1cm] (\name) at (\angle:1cm) {$\text$};

	\foreach \from/\to in {1/2,2/3,3/4,4/1}{
		\draw (S1-\from) -- (S1-\to);
		\draw	(S2-\from) -- (S2-\to);
	}


	\foreach \name/\angle/\color/\text in {L-1/180/vw/Q,L-2/0/vb/R}
		\node[\color,xshift=10.5cm,yshift=0cm] (\name) at (\angle:1cm) {$\text$};

	\draw [->] (L-1) to [bend right=37] (L-2);
	\draw [->] (L-2) to [bend right=37] (L-1);

\end{tikzpicture}

\caption{The image of the permutation $\pi = (1 \ 18 \ 16 \ 8 \ 9)(2 \ 17 \ 10)(3 \ 15 \ 7 \ 11)(4 \ 14 \ 6 \ 12)(5 \ 13)$ under our bijection. White vertices came from block $A_1$ and grey vertices came from block $A_2$. The labels $A$ through $R$ are only for the later convenience of referring to specific vertices.}
\label{GR-bijection-example}

\end{figure}


Now suppose that we wanted to go in the reverse direction. Starting with the ornament depicted in Figure \ref{GR-bijection-example}, how do we know what permutation it came from? We need to replace $A$ through $R$ with the integers $1$ through $18$ so that the resulting permutation descends in $A_1$ and ascends in $A_2$. The colors of each vertex (white or grey, denoting block $1$ or block $2$, respectively) narrow the possibilities down somewhat, as they tell us whether a vertex comes from $A_1$ or $A_2$.

We start out by trying to determine the relative ordering of pairs of vertices. We know immediately (since all elements of $A_1$ come before all elements of $A_2$) that $A$, $D$, $F$, $I$, $K$, $M$, $O$, $Q$ all come before $B$, $C$, $E$, $G$, $H$, $J$, $L$, $N$, $P$, $R$. We also know, for example, that $E < B$ since $\pi(E) < \pi(B)$ (as $\pi(E) = A$, $\pi(B) = C$, and $A < C$) and $B$ and $E$ come from $A_2$, which must ascend.

We can then determine that $A < D$ since $\pi(A) = B$, $\pi(D) = E$, and $B > E$. Continuing, we see that $C > E$ since $\pi(C) = D$, $\pi(E) = A$, and $D > A$.

We can get a lot of information by making these sorts of comparisons, but we would like something a bit more methodical so that we can piece together all the information at the end. We can do this by starting at two vertices and ``looking forward'' along the paths from those vertices until the paths differ. For example, we could determine that $C > E$ as follows:

Starting from $C$, we see vertices colored grey, white, grey, white. For clarity, we will call these vertices $c_1$, $c_2$, $c_3$, and $c_4$. Starting from $E$, we see vertices colored grey, white, grey, grey. We will call these vertices $e_1$, $e_2$, $e_3$, and $e_4$. Since $c_4$ is white and $e_4$ is grey, we must have $c_4 < e_4$. Then $c_3$ and $e_3$ both come from $A_2$, so $c_3 < e_3$. Then $c_2 > e_2$, and $c_1 > e_1$. Since $C = c_1$ and $E = e_1$, we deduce that $C > E$.

This sort of logic is captured more formally in the following lemma and its corollary.

\begin{lemma}
\label{walk-lemma}
Given a vertex $v$ of a necklace, define the sequence $W(v) = \{w_0(v),w_1(v),\ldots\}$ by $w_0(v) = v$, $w_{i+1}(v) = s(w_i(v))$, where $s(x)$ is the successor of $x$ in the necklace. Thus $w_0,w_1,\ldots$ is the sequence of colors one encounters if one starts at the vertex $v$ and walks along the cycle containing $v$.

If two vertices $v$ and $v'$ have sequences of colors that agree through $w_{l-1}$, then the order of $v$ and $v'$ is determined by the order of $w_{l}(v)$ and $w_{l}(v')$. In fact, if $\{w_1,\ldots,w_{l-1}\}$ has an even number of vertices from descending blocks, then $v$ and $v'$ come in the same order as $w_{l}(v)$ and $w_{l}(v')$. Otherwise, they come in the opposite order.
\end{lemma}

\begin{corollary}
\label{walk-corollary}
Define another sequence $A(v) = \{a_0(v),a_1(v),\ldots\}$ by $a_i(v) = (-1)^{r_i(v)} w_i(v)$, where $r_i(v)$ is the number of vertices in $\{w_0(v),\ldots,w_{i-1}(v)\}$ that come from descending blocks. Then $v$ and $w$ come in the same order as $A(v)$ and $A(w)$, if we consider the latter pair in the lexicographic order.
\end{corollary}

We call $W(v)$ the \emph{walk} from $v$ and $A(v)$ the \emph{signed walk} from $v$. We will prove Lemma \ref{walk-lemma} and Corollary \ref{walk-corollary} later in this section.

\begin{table}[b]
\caption{The first 7 terms of $A(v)$ for $v = A,\ldots,R$. We have ordered the entries lexicographically by $A(v)$.}
\label{GR-bijection-table}

\centering

\begin{tabular}{|| l | l || l | l || l | l ||}
\hline
A & $1, \mi 2, \mi 2, \mi 1, \pl 2, \pl 1, \mi 2$ & Q & $1, \mi 2, \mi 1, \pl 2, \pl 1, \mi 2, \mi 1$ & N & $2, \pl 1, \mi 2, \mi 1, \pl 2, \pl 1, \mi 2$ \\
F & $1, \mi 2, \mi 2, \mi 1, \pl 2, \pl 2, \pl 1$ & D & $1, \mi 2, \mi 1, \pl 2, \pl 2, \pl 1, \mi 2$ & P & $2, \pl 1, \mi 2, \mi 1, \pl 2, \pl 1, \mi 2$ \\
I & $1, \mi 2, \mi 1, \pl 2, \pl 1, \mi 2, \mi 1$ & E & $2, \pl 1, \mi 2, \mi 2, \mi 1, \pl 2, \pl 1$ & R & $2, \pl 1, \mi 2, \mi 1, \pl 2, \pl 1, \mi 2$ \\
K & $1, \mi 2, \mi 1, \pl 2, \pl 1, \mi 2, \mi 1$ & H & $2, \pl 1, \mi 2, \mi 2, \mi 1, \pl 2, \pl 2$ & C & $2, \pl 1, \mi 2, \mi 1, \pl 2, \pl 2, \pl 1$ \\
M & $1, \mi 2, \mi 1, \pl 2, \pl 1, \mi 2, \mi 1$ & J & $2, \pl 1, \mi 2, \mi 1, \pl 2, \pl 1, \mi 2$ & G & $2, \pl 2, \pl 1, \mi 2, \mi 2, \mi 1, \pl 2$ \\
O & $1, \mi 2, \mi 1, \pl 2, \pl 1, \mi 2, \mi 1$ & L & $2, \pl 1, \mi 2, \mi 1, \pl 2, \pl 1, \mi 2$ & B & $2, \pl 2, \pl 1, \mi 2, \mi 1, \pl 2, \pl 2$ \\
\hline
\end{tabular}

\end{table}

Table \ref{GR-bijection-table} gives the sequences $A(v)$ for $v = A,\ldots,R$. From this, we can determine the values of some of $A,\ldots,R$. We know that $A = 1$, $B = 18$, $C = 16$, $D = 8$, $E = 9$, $F = 2$, $G = 17$, and $H = 10$. The only uncertainty is what happens with $I$ through $R$, although we know at least that $\{I,K,M,O,Q\} = \{3,4,5,6,7\}$ and $\{J,L,N,P,R\} = \{11,12,13,14,15\}$.

By symmetry, we can assume that either $I$ or $Q$ is equal to $3$. Let's start by assuming that $Q$ is equal to $3$. Then $R$ must be $11$, since its successor is the smallest among the successors of $J,L,N,P,R$ and $A_2$ ascends. On the other hand, $R$ must be $15$, since it is the successor of $Q$, $Q$ is the smallest among $I,K,M,O,Q$, and $A_1$ descends. We have thus reached a contradiction, so $I$ \emph{must} be $3$.

Having determined the value of $I$, we see that $L$ must be $11$, $K$ must be $7$, and $J$ must be $15$. So we are left with assigning $M,O,Q$ to $4,5,6$ and $N,P,R$ to $12,13,14$. Again by symmetry we can assume that either $M$ or $Q$ is $4$. By the same logic as before, we can show that $Q$ cannot equal $4$ and so $M$ must be $4$. This forces $P = 12$, $O = 6$, and $N = 14$, which in turn forces $Q = 5$ and $R = 13$.

Putting this all together, we get the permutation depicted in Figure \ref{GR-bijection-example-2}, which is the permutation $\pi$ that we started with. 


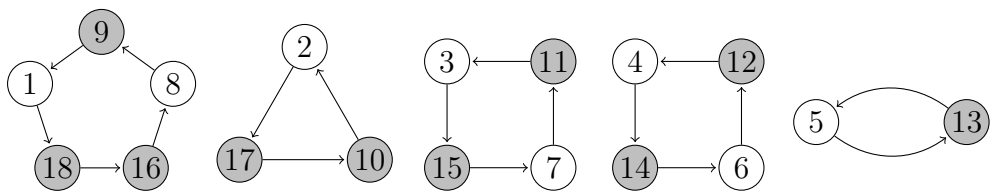
\begin{figure}[b]

\centering

\begin{tikzpicture}[shorten >=1pt,->]
	\tikzstyle{vw}=[circle,draw=black!100,fill=white,minimum size=17pt,inner sep=0pt]
	\tikzstyle{vb}=[circle,draw=black!100,fill=black!25,minimum size=17pt,inner sep=0pt]

	\foreach \name/\angle/\color/\text in {P-1/162/vw/1,
				P-2/234/vb/18, P-3/306/vb/16,	P-4/18/vw/8, P-5/90/vb/9}
		\node[\color,xshift=0cm,yshift=0.2cm] (\name) at (\angle:1cm) {$\text$};

	\foreach \from/\to in {1/2,2/3,3/4,4/5,5/1}
		\draw (P-\from) -- (P-\to);
	
	\foreach \name/\angle/\color/\text in {T-1/90/vw/2,
				T-2/210/vb/17, T-3/330/vb/10}
		\node[\color,xshift=2.7cm,yshift=0cm] (\name) at (\angle:1cm) {$\text$};

	\foreach \from/\to in {1/2,2/3,3/1}
		\draw (T-\from) -- (T-\to);


	\foreach \name/\angle/\color/\text in {S1-1/135/vw/3,
				S1-2/225/vb/15, S1-3/315/vw/7, S1-4/45/vb/11}
		\node[\color,xshift=5.3cm,yshift=0.1cm] (\name) at (\angle:1cm) {$\text$};

	\foreach \name/\angle/\color/\text in {S2-1/135/vw/4,
				S2-2/225/vb/14, S2-3/315/vw/6, S2-4/45/vb/12}
		\node[\color,xshift=7.8cm,yshift=0.1cm] (\name) at (\angle:1cm) {$\text$};

	\foreach \from/\to in {1/2,2/3,3/4,4/1}{
		\draw (S1-\from) -- (S1-\to);
		\draw	(S2-\from) -- (S2-\to);
	}


	\foreach \name/\angle/\color/\text in {L-1/180/vw/5,L-2/0/vb/13}
		\node[\color,xshift=10.5cm,yshift=0cm] (\name) at (\angle:1cm) {$\text$};

	\draw [->] (L-1) to [bend right=37] (L-2);
	\draw [->] (L-2) to [bend right=37] (L-1);

\end{tikzpicture}

\caption{The pre-image of the ornament in Figure \ref{GR-bijection-example} under our bijection. This yields the permutation $\pi = 18 \ 17 \ 15 \ 14 \ 13 \ 12 \ 11 \ 9 \ 1 \ 2 \ 3 \ 4 \ 5 \ 6 \ 7 \ 8 \ 10 \ 16$.}
\label{GR-bijection-example-2}

\end{figure}


We now state the bijection formally. The \emph{fundamental period} of a necklace is the smallest contiguous subsequence $P$ of the necklace such that the necklace can be obtained by concatenating $r$ copies of $P$ for some $r$. In this case, the necklace is said to be \emph{$r$-repeating}. Call an ornament \emph{$A$-compatible} if its vertices are labeled by $\{1,\ldots,k\}$ and exactly $a_i$ vertices are labeled by $i$.

\begin{theorem}
\label{GR-bijection}

There is an injection from the $(A,S)$-permutations to the $A$-compatible ornaments. Furthermore, this map preserves cycle structure. The image of the map is all $A$-compatible ornaments satisfying the following three conditions.

\begin{enumerate}

\item If the fundamental period of a necklace contains an even number of vertices from descending blocks, then the necklace is $1$-repeating.

\item If the fundamental period of a necklace contains an odd number of vertices from descending blocks, then the necklace is either $1$-repeating or $2$-repeating.

\item If a necklace contains an odd number of vertices from descending blocks, then there are no other necklaces identical to it in the ornament.

\end{enumerate}

\end{theorem}


\begin{proof}

We begin by describing the bijection formally. We take a permutation $\pi$, write it as a product of disjoint cycles, and replace each element of each cycle of $\pi$ by the block it belongs to. This leaves us with an $A$-compatible ornament.


We note that the intuition from the example above can be formalized to provide a complete proof of Theorem \ref{GR-bijection}. However, we can use a technical device to make it more clear where the conditions in Theorem \ref{GR-bijection} come from, so we will take this approach instead.


To get the reverse map, we need to take an ornament $\omega$ and then replace, for each $i$, the vertices colored $i$ by the elements of $A_i$. In this way, we go from an ornament to a permutation. We can equivalently think of this as ordering the vertices of $\omega$ so that the elements of $A_1$ come first, then the elements of $A_2$, etc. Our map is injective if and only if, for any ornament $\omega$, there is at most one way to do this that yields an $(A,S)$-permutation.


As in Corollary \ref{walk-corollary}, let $A(v)$ be the signed walk from $v$. Let $P(v) := \{ v' \mid A(v') = A(v) \}$. We call $P(v)$ the \emph{packet} of $v$. We note that $A(v) = A(v')$ if and only if $W(v) = W(v')$ (this observation is justified later, in Lemma \ref{walk-equals-signed-walk}).


Given a sequence $X = x_0,x_1,x_2,\ldots$, define $S(X)$ to be the sequence $x_1,x_2,\ldots$. Also, given a vertex $v$ of an ornament $\omega$, let $s(v)$ denote the successor of $v$ in the relevant cycle. Thus $W(s(v)) = S(W(v))$ and $A(s(v)) = \pm S(A(v))$. Given a packet $P = P(v)$, let $S(P) := P(s(v))$. We call $S(X)$ the \emph{successor sequence} of $X$, $s(v)$ the \emph{successor vertex} of $v$, and $S(P)$ the \emph{successor packet} of $P$.

Given a packet $P$, let $O(P)$ be the orbit of $P$ under the map $S$. We think of $O(P)$ as a necklace of packets. In general, we will call a $1$-repeating necklace of packets an \emph{orbit}. For any packet $P$, $O(P)$ is automatically $1$-repeating, since if it were $r$-repeating for $r > 1$ the packets at corresponding locations in different periods would have the same walks and thus should have been part of the same packet to begin with.

We will call a set of orbits a \emph{template}. We can go from an $A$-compatible ornament to a template by grouping the vertices into packets, then grouping the packets into orbits.

We observe that we can recover a template solely from the walks from each vertex, as opposed to the ornament itself. Furthermore, a template comes from an $A$-compatible ornament if and only if the following two conditions hold:

\begin{itemize}

\item Within each orbit, every packet has the same size.

\item No two orbits have the same color sequence (otherwise they could be grouped into a larger orbit).

\item The number of vertices colored $i$ is $a_i$.

\end{itemize}

If these conditions hold, we will call the template \emph{$A$-compatible}. We note that an $A$-compatible template comes from many different $A$-compatible ornaments, one for each way of partitioning the orbits up into necklaces of vertices (to get an $r$-repeating necklace, we take $r$ vertices from each packet in an orbit). The remarkable thing, and the heart of Theorem \ref{GR-bijection}, is that only one of these corresponds to an $(A,S)$-permutation. This is what we will show next.


\def\yscale{0.87}
\def\xscale{1.0}

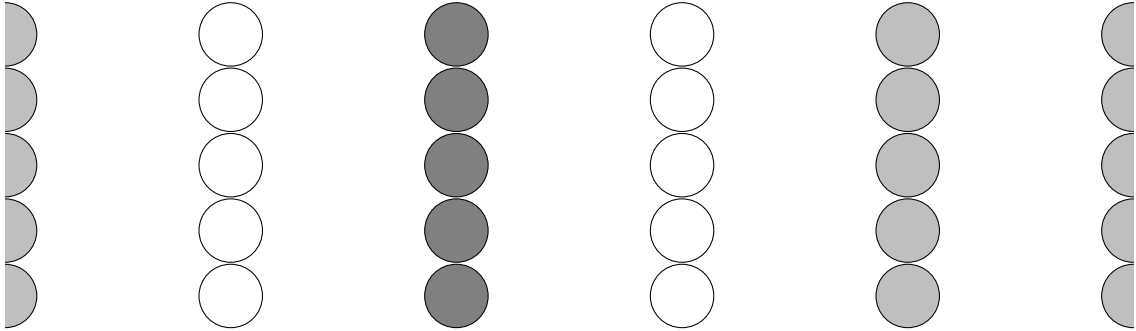
\begin{figure}[p!]

\centering

\begin{tikzpicture}[-triangle 45]
	\tikzstyle{vw}=[circle,draw=black,fill=black!00,minimum size=24pt]
	\tikzstyle{vg}=[circle,draw=black,fill=black!25,minimum size=24pt]
	\tikzstyle{vb}=[circle,draw=black,fill=black!50,minimum size=24pt]

\clip (0,0) rectangle (\xscale*15,\yscale*6);

\foreach \x/\xe/\color in {0/3/vg,3/6/vw,6/9/vb,9/12/vw,12/15/vg,15/0/vg}{
	\foreach \y in {1,2,3,4,5}{
		\node[\color] (P-\x-\y)at (\xscale*\x,\yscale*\y) {};
	}
}

\end{tikzpicture}

\caption{A template with a single orbit consisting of $5$ packets, each with $5$ vertices. The first and last half-vertex are the same. Light grey indicates block $1$, white indicates block $2$, and dark grey indicates block $3$, so $A = (10,10,5)$. Also, $S = \{1,3\}$, so blocks $1$ and $3$ descend while block $2$ ascends.}
\label{GR-template}

\end{figure}

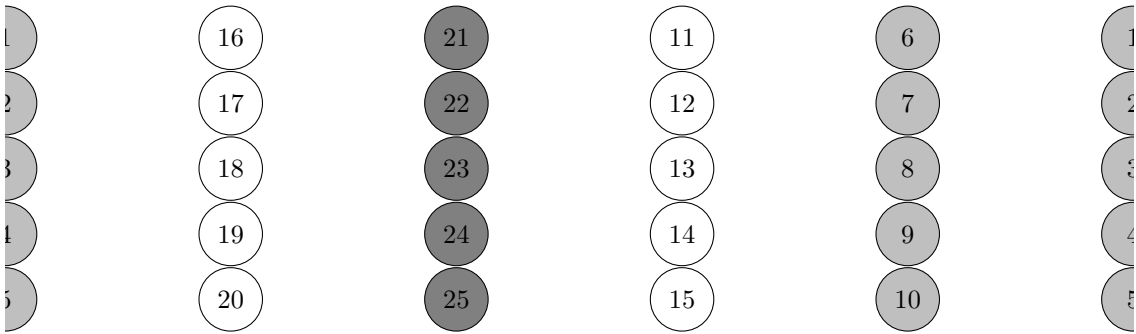
\begin{figure}[p!]

\centering

\begin{tikzpicture}[-triangle 45]
	\tikzstyle{vw}=[circle,draw=black,fill=black!00,minimum size=24pt]
	\tikzstyle{vg}=[circle,draw=black,fill=black!25,minimum size=24pt]
	\tikzstyle{vb}=[circle,draw=black,fill=black!50,minimum size=24pt]

\clip (0,0) rectangle (15,\yscale*6);

\foreach \y/\text in {1/5,2/4,3/3,4/2,5/1}{
	\node[vg] (P-0-\y) at (\xscale*0,\yscale*\y) {\footnotesize{$\text$}};
}

\foreach \y/\text in {1/20,2/19,3/18,4/17,5/16}{
	\node[vw] (P-3-\y) at (\xscale*3,\yscale*\y) {\footnotesize{$\text$}};
}

\foreach \y/\text in {1/25,2/24,3/23,4/22,5/21}{
	\node[vb] (P-6-\y) at (\xscale*6,\yscale*\y) {\footnotesize{$\text$}};
}

\foreach \y/\text in {1/15,2/14,3/13,4/12,5/11}{
	\node[vw] (P-9-\y) at (\xscale*9,\yscale*\y) {\footnotesize{$\text$}};
}

\foreach \y/\text in {1/10,2/9,3/8,4/7,5/6}{
	\node[vg] (P-12-\y) at (\xscale*12,\yscale*\y) {\footnotesize{$\text$}};
}

\foreach \y/\text in {1/5,2/4,3/3,4/2,5/1}{
	\node[vg] (P-15-\y) at (\xscale*15,\yscale*\y) {\footnotesize{$\text$}};
}

\end{tikzpicture}

\caption{The unique way of numbering the vertices in the template of Figure \ref{GR-template} to get an $(A,S)$-permutation, based on Corollary \ref{walk-corollary}.}
\label{GR-numbering}

\end{figure}

\begin{figure}[p!]

\centering

\begin{tikzpicture}[-triangle 45]
	\tikzstyle{vw}=[circle,draw=black,fill=black!00,minimum size=24pt]
	\tikzstyle{vg}=[circle,draw=black,fill=black!25,minimum size=24pt]
	\tikzstyle{vb}=[circle,draw=black,fill=black!50,minimum size=24pt]

\clip (0,0) rectangle (15,\yscale*6);

\foreach \y/\text in {1/5,2/4,3/3,4/2,5/1}{
	\node[vg] (P-0-\y) at (\xscale*0,\yscale*\y) {\footnotesize{$\text$}};
}

\foreach \y/\text in {1/20,2/19,3/18,4/17,5/16}{
	\node[vw] (P-3-\y) at (\xscale*3,\yscale*\y) {\footnotesize{$\text$}};
}

\foreach \y/\text in {1/25,2/24,3/23,4/22,5/21}{
	\node[vb] (P-6-\y) at (\xscale*6,\yscale*\y) {\footnotesize{$\text$}};
}

\foreach \y/\text in {1/15,2/14,3/13,4/12,5/11}{
	\node[vw] (P-9-\y) at (\xscale*9,\yscale*\y) {\footnotesize{$\text$}};
}

\foreach \y/\text in {1/10,2/9,3/8,4/7,5/6}{
	\node[vg] (P-12-\y) at (\xscale*12,\yscale*\y) {\footnotesize{$\text$}};
}

\foreach \y/\text in {1/5,2/4,3/3,4/2,5/1}{
	\node[vg] (P-15-\y) at (\xscale*15,\yscale*\y) {\footnotesize{$\text$}};
}

\foreach \x/\xe/\color in {0/3/vg,3/6/vw,6/9/vb,9/12/vw,12/15/vg,15/0/vg}{
	\ifthenelse{\x < 15}{
		\ifthenelse{\equal{\color}{vw}}{
			\foreach \y/\ye in {1/1,2/2,3/3,4/4,5/5}{
				\draw (P-\x-\y) -- (P-\xe-\ye);
			}
		}
		{
			\foreach \y/\ye in {1/5,2/4,3/3,4/2,5/1}{
				\draw (P-\x-\y) -- (P-\xe-\ye);
			}
		}
	}{}

}

\end{tikzpicture}

\caption{The unique way of choosing successors for the numbered template in Figure \ref{GR-numbering} to yield an $(A,S)$-permutation. The successors are indicated by arrows. Observe that we end up with the $5$-cycle $(3 \ 18 \ 23 \ 13 \ 8)$ and two $10$-cycles.}
\label{GR-successors}

\end{figure}
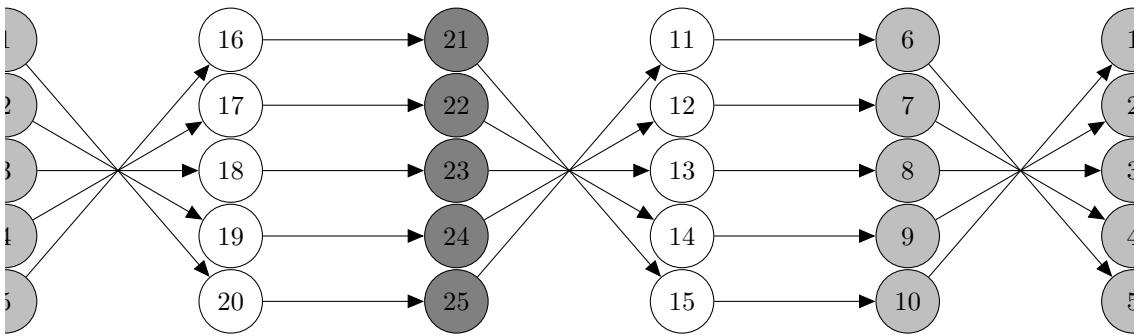

To get a permutation from an $A$-compatible template, we must first label the vertices by the integers $1$ through $n$, then choose the successor of each vertex. There is the constraint that the vertices labeled $i$ are labelled by the elements of $A_i$ and that $s(v)$ lie in $S(P(v))$, but other than that, we can do as we wish. Figures \ref{GR-template}, \ref{GR-numbering}, and \ref{GR-successors} illustrate how to do this for a given template to yield an $(A,S)$-permutation. We describe the general process in the next few paragraphs, as well as prove that it is the unique way to get an $(A,S)$-permutation.

By Corollary \ref{walk-corollary}, there is only one labeling of the vertices that can yield an $(A,S)$-permutation. It is obtained by first listing the vertices $v_1,\ldots,v_n$ of the template so that if $i < j$ then $A(v_i) < A(v_j)$; then labeling $v_i$ with the integer $i$ (ties in $A(v)$ are irrelevant here because all vertices with the same walk are symmetric).

There is also a unique way to pick the successors of each vertex to get an $(A,S)$-permutation. If a packet has vertices from an ascending block, then the successors of the vertices should be ordered in the same way as the vertices themselves (and lie in the successor packet). If a packet has vertices from a descending block, then the successors of the vertices should be ordered in the opposite way that the vertices themselves are ordered (and again lie in the successor packet). This constraint uniquely determines the successors of each vertex, and we also can see that this constraint is sufficient to get an $(A,S)$-permutation, since the prior lexicographic ordering ensures that we only have to worry about vertices within the same packet. 

Now that we have constructed a unique $(A,S)$-permutation, we look at its cycle structure.

If an orbit has an even number (say $d$) of packets from descending blocks and $x$ packets in total, and each packet has size $y$, then that orbit will add $y$ cycles, all of length $x$, to the permutation.

If an orbit has an odd number (say $d$) of packets from descending blocks, then we should get cycles of size $2x$ since one needs to go around the orbit twice to get back to the original vertex. The exception is if $y$ is odd, in which case there is one cycle of length $x$ that contains the vertices in the middle of each packet.

We thus see that, for each orbit, there is a unique set of cycles it could come from to yield an $(A,S)$-permutation. Furthermore, the parameter $d$ in the two paragraphs above is the number of vertices that the fundamental period of a cycle has from descending blocks, and the parameter $y$ is the total number of times that the fundamental period occurs in all cycles (we consider the fundamental period to occur twice in a $2$-repeating cycle). The cycles we end up with correspond exactly to the conditions of Theorem \ref{GR-bijection}: if $d$ is even, then we get $y$ $1$-repeating cycles that are all isomorphic; if $d$ is odd, then we get $\lfloor \frac{y}{2} \rfloor$ isomorphic $2$-repeating cycles, and possibly a lone $1$-repeating cycle if $y$ is also odd. We are therefore done.
\end{proof}


We observe that, in the course of the proof, we also got a bijection between $(A,S)$-permutations and $A$-compatible packets. On the other hand, $A$-compatible packets are in bijection with $A$-compatible ornaments where every necklace is $1$-repeating (by replacing an orbit where every packet has size $y$ by $y$ necklaces whose vertices have the same colors as one encounters when going through the orbit). We will refer to $A$-compatible ornaments where every necklace is $1$-repeating as \emph{$A$-good} ornaments.

For later convenience, we describe a bijection between $(A,S)$-permutations and $A$-good ornaments. We can think of this bijection as what we get if we take the map from $(A,S)$-permutations to $A$-compatible packets, then go from $A$-compatible packets to $A$-good ornaments. However, while following these maps is a bit tricky, the composite map turns out to be quite simple.


As an example, take the permutation depicted in Figure \ref{GR-bijection-example-2}. Under the bijection of Theorem \ref{GR-bijection}, this permutation gets sent to the ornament in Figure \ref{GR-bijection-example}. However, not all necklaces in this ornament are $1$-repeating. In particular, the two $4$-cycles are both $2$-repeating. What we will do is replace each of the $4$-cycles with two $2$-cycles, where the $2$-cycles have the same fundamental period as the $4$-cycles. This leaves us with seven necklaces: a $5$-cycle, a $3$-cycle, and five $2$-cycles, and each necklace is $1$-repeating.


The correspondence in general is documented in Proposition \ref{auxiliary-bijection}.

\begin{proposition}
\label{auxiliary-bijection}
There is a bijection between $(A,S)$-permutations and $A$-good ornaments.
\end{proposition}

\begin{proof}
The bijection in the forward direction is as follows: first take the map in Theorem \ref{GR-bijection}. Then replace every $2$-repeating cycle by two $1$-repeating cycles with the same fundamental period.

The bijection in the reverse direction is as follows: for every necklace $\nu$ with an odd number of vertices from descending blocks, let $c(\nu)$ denote the number of times that necklace appears in the ornament. If $c(\nu)$ is odd, then replace the $c(\nu)$ necklaces with a single necklace identical to $\nu$ and $\frac{c(\nu)-1}{2}$ necklaces that are $2$-repeating and have the same fundamental period as $\nu$. If $c(\nu)$ is even, then replace the $c(\nu)$ necklaces with $\frac{c(\nu)}{2}$ necklaces that are $2$-repeating and have the same fundamental period as $\nu$.
\end{proof}

In the next section we will count the $(A,S)$-derangements. For this, the following proposition will be helpful.

\begin{proposition}
\label{derangement-bijection}
The $(A,S)$-derangements are in bijection with $A$-good ornaments with no $1$-cycles from ascending blocks and an even number of $1$-cycles from each descending block.
\end{proposition}

\begin{proof}
By applying Proposition \ref{auxiliary-bijection}, we get an injection from the $(A,S)$-derangements into the $A$-good ornaments. For an $A$-good ornament to correspond to a permutation with no fixed points, it must have no $1$-cycles from ascending blocks. It is permissible for the ornament to have $1$-cycles from descending blocks, however, since a pair of $1$-cycles from a descending block corresponds to a $2$-cycle in the actual $(A,S)$-permutation. We can therefore have any even number of $1$-cycles from each descending block. The $(A,S)$-derangements are thus in bijection with $A$-good ornaments with no $1$-cycles in ascending blocks and an even number of $1$-cycles in each descending block, as was to be shown.
\end{proof}

We record the following corollary for use in another paper \cite{Ste2}.

\begin{corollary}
\label{derangement-bijection-corollary}
The $(A,S)$-derangements are in bijection with $A$-compatible ornaments satisfying the following properties:

\begin{itemize}

\item Every cycle is either $1$-repeating or $2$-repeating.

\item The only $2$-repeating cycles are monochromatic $2$-cycles from a descending block.

\item There are no $1$-cycles.

\end{itemize}

\end{corollary}

\begin{proof}
We can apply Proposition \ref{derangement-bijection}, then replace every pair of $1$-cycles from a descending block with a $2$-cycle from the same block.
\end{proof}

We have a few loose ends to tie up, which are to prove Lemma \ref{walk-lemma} and Corollary \ref{walk-corollary} as well as state and prove Lemma \ref{walk-equals-signed-walk}.

\begin{lemma}
\label{walk-equals-signed-walk}
Let $v$ and $v'$ be two vertices. Their walks $W(v)$ and $W(v')$ agree up through $w_i$ if and only if their signed walks $A(v)$ and $A(v')$ agree up through $a_i$.
\end{lemma}

\begin{proof}
If their signed walks agree up through $a_i$, their walks must agree up through $w_i$, since $a_i = \pm w_i$ and $w_i > 0$ always.

Now suppose their walks agree up through $w_i$. Then $r_j(v) = r_j(v')$ for all $j \leq i+1$ and $w_j(v) = w_j(v')$ for all $j \leq i$, so $(-1)^{r_j(v)} w_j(v) = (-1)^{r_j(v')} w_j(v')$ for all $j \leq i$. This is the same as saying that $a_j(v) = a_j(v')$ for all $j \leq i$, so we are done.
\end{proof}

\begin{proof}[Proof of Lemma \ref{walk-lemma}]
Let $r_l(v)$ be, as before, the number of vertices in $\{w_0(v),\ldots,w_{l-1}(v)\}$ that come from descending blocks. We are trying to prove that if the walks from $v$ and $v'$ agree through $w_{l-1}$, then $v$ and $v'$ come in the same order as $(-1)^{r_l(v)} w_l(v)$ and $(-1)^{r_l(v')} w_l(v')$.

We proceed by induction on $l$. In the base case $l = 1$, the result is a consequence of the fact that $v$ and $v'$ come from the same block, and if that block is ascending then $v$ and $v'$ are in the same order as their successors, whereas if it is descending they are in the opposite order.

Now suppose that $v$ and $v'$ have sequences of colors that agree through $w_l$. Then they also agree through $w_{l-1}$, so by the inductive hypothesis $v$ and $v'$ come in the same order as $(-1)^{r_l(v)} w_l(v)$ and $(-1)^{r_l(v')} w_l(v')$ since $w_l(v)$ and $w_l(v')$ have the same color. By taking the case $l = 1$ applied to $w_l(v)$ and $w_l(v')$, we know that $w_l(v)$ and $w_l(v')$ come in the same order as $(-1)^{r_1(w_l(v))} w_{l+1}(v)$ and $(-1)^{r_1(w_l(v'))} w_{l+1}(v')$. Hence $v$ and $v'$ come in the same order as $(-1)^{r_l(v)+r_1(w_l(v))} w_{l+1}(v)$ and $(-1)^{r_l(v')+r_1(w_l(v'))} w_{l+1}(v')$. Since $r_l(v)+r_1(w_l(v)) = r_{l+1}(v)$, the lemma follows.
\end{proof} 

\begin{proof}[Proof of Corollary \ref{walk-corollary}]
Suppose that $A(v) < A(v')$ lexicographically. Then there exists an $l$ such that $A(v)$ and $A(v')$ first differ in the $l$th position, so the signed walks from $v$ and $v'$ agree through $a_{l-1}$. By Lemma \ref{walk-equals-signed-walk}, this means that the walks from $v$ and $v'$ agree through $w_{l-1}$, so $v$ and $v'$ come in the same order as $a_l(v)$ and $a_l(v')$. But $a_l(v) < a_l(v')$ by assumption, so $v < v'$, as was to be shown.
\end{proof}

\section{Proofs of the Corollaries}
\label{corollaries}

\begin{corollary}
\label{EFW-symmetric}
Let $\sigma$ be a permutation of $\{1,\ldots,k\}$ and let $\mathcal{C}$ be a conjugacy class in $S_n$. The number of $(a_1,\ldots,a_k,S)$-permutations in $\mathcal{C}$ is the same as the number of $(a_{\sigma(1)},\ldots,a_{\sigma(k)},\sigma(S))$-permutations in $\mathcal{C}$.
\end{corollary}

\begin{proof}
The image of the map in Theorem \ref{GR-bijection} doesn't distinguish between the blocks.
\end{proof}

\begin{corollary}
\label{EFW-enumerate}
The number of $(A,S)$-derangements is the coefficient of $x_1^{a_1}\cdots x_k^{a_k}$ in 

\begin{equation}
\label{EFW-GF}
\frac{1}{1-x_1-\cdots-x_k} \left(\frac{\prod_{i \not\in S} (1-x_i)}{\prod_{i \in S} (1+x_i)}\right).
\end{equation}

\noindent Let $l_m = a_m$ if $m \in S$ and let $l_m = 1$ otherwise. The number of $(A,S)$-derangements is also

\begin{equation}
\label{EFW-PIE}
\sum_{0 \leq b_m \leq l_m, m=1,\ldots,k} (-1)^{\sum b_i} \binom{\sum (a_i-b_i)}{a_1-b_1,\ldots,a_k-b_k}.
\end{equation}

\end{corollary}

\begin{proof}
As in Section \ref{bijections}, we will refer to an $A$-compatible ornament where every necklace is $1$-repeating as an \emph{$A$-good ornament}.

First note that 

\[
\frac{1}{1-x_1-\cdots-x_k} = \sum_{n=0}^{\infty} \left(\sum_{i=1}^k x_i\right)^n = \sum_{c_1,\ldots,c_k=0}^{\infty} \binom{c_1+\cdots+c_k}{c_1,\ldots,c_k} x_1^{c_1}\cdots x_k^{c_k}.
\]

From here it is easy to see that the $x_1^{a_1}\cdots x_k^{a_k}$ coefficient in (\ref{EFW-GF}) is equal to the sum given in (\ref{EFW-PIE}). It thus suffices to establish that (\ref{EFW-PIE}) enumerates the $(A,S)$-derangements.

By Proposition \ref{derangement-bijection}, the $(A,S)$-derangements are in bijection with the $A$-good ornaments with no $1$-cycles in ascending colors and an even number of $1$-cycles in each descending color.

Note that the number of $(a_1,\ldots,a_k)$-good ornaments is $\binom{a_1+\cdots+a_k}{a_1,\ldots,a_k}$. This is because these ornaments are in bijection with the $(a_1,\ldots,a_k)$-ascending permutations by Theorem \ref{GR-bijection}. There are $\binom{a_1+\cdots+a_k}{a_1,\ldots,a_k}$ $(a_1,\ldots,a_k)$-ascending permutations because, once we determine the set of permutation values within each block, there is exactly one way to order them to be increasing.

Also, the number of $(a_1,\ldots,a_k)$-good ornaments with at least $b_i$ $1$-cycles of color $i$ is $\binom{(a_1-b_1)+\cdots+(a_k-b_k)}{a_1-b_1,\ldots,a_k-b_k}$. This is because they are in bijection with the $(a_1-b_1,\ldots,a_k-b_k)$-good ornaments (the bijection comes from removing, for each $i$, $b_i$ of the $1$-cycles of color $i$).

Now if $f(b_1,\ldots,b_k)$ is the number of $(a_1,\ldots,a_k)$-good ornaments with at least $b_i$ $1$-cycles of color $i$, then a standard inclusion-exclusion argument shows that the number of ornaments with an even number of $1$-cycles in descending colors and no $1$-cycles in ascending colors is

\[
\sum_{0 \leq b_m \leq l_m, m=1,\ldots,k} (-1)^{\sum b_i} f(b_1,\ldots,b_k)
\]

\noindent where $l_m = a_m$ if $m \in S$ and $l_m = 1$ if $m \not\in S$. Since we know that $f(b_1,\ldots,b_k) = \binom{(a_1-b_1)+\cdots+(a_k-b_k)}{a_1-b_1,\ldots,a_k-b_k}$, (\ref{EFW-PIE}) follows.
\end{proof}

\begin{corollary}
\label{GR-4.1}
Associate to each conjugacy class $\mathcal{C}$ of $S_n$ a partition $\lambda$ of $n$ based on cycle structure.

The number of $(A,S)$-permutations in $\mathcal{C}$ is the same if we replace $S$ by $\{1,\ldots,k\}\backslash S$, assuming that all odd parts of $\lambda$ are distinct and $\lambda$ has no parts congruent to $2$ mod $4$.
\end{corollary}

\begin{proof}
We will take an ornament that satisfies the conditions of Theorem \ref{GR-bijection}, then show that it still satisfies the conditions of Theorem \ref{GR-bijection} if we make each ascending block a descending block and vice versa. This would provide an injection from the $(A,S)$-permutations in $\mathcal{C}$ and the $(A,\{1,\ldots,k\}\backslash S)$-permutations in $\mathcal{C}$. Since taking the complement of $S$ twice yields $S$ again, this is sufficient.

Suppose we have an ornament $\omega$ that satisfies the conditions of Theorem \ref{GR-bijection}. Then (i) every necklace with an even number of vertices from descending blocks in its fundamental period is $1$-repeating, (ii) every necklace with an odd number of vertices from descending blocks in its fundamental period is either $1$-repeating or $2$-repeating, and (iii) no two necklaces with an odd number of vertices from descending blocks are isomorphic.

If a necklace has an even number of total vertices, then the conditions on $\lambda$ ensure that the number of vertices in the cycle is divisible by $4$. Since every necklace is at most $2$-repeating, this means that the size of the fundamental period must be even. In this case, the number of vertices from ascending and descending blocks in the fundamental period has the same parity. Therefore, whether the necklace satisfies the hypotheses of (i), (ii), and (iii) remains unchanged when we replace $S$ by its complement; since we are left with the same necklace, whether that necklace satisfies the conclusions of (i), (ii), and (iii) also remains unchanged.

If a necklace has an odd number of total vertices, then the conditions on $\lambda$ imply that it is the only necklace with that many vertices and thus cannot be isomorphic to any other necklace. Thus the conclusion of (iii) is automatically satisfied. The necklace also cannot be $2$-repeating, since it has an odd number of total vertices, so the conclusions of (i) and (ii) combine to say that, in all cases, the necklace must be $1$-repeating. This condition is independent of $S$, so whether this necklace satisfies the conditions imposed by (i), (ii), and (iii) does not change if we replace $S$ by its complement.

We have shown that an ornament satisfying the conditions of Theorem \ref{GR-bijection} will still do so if we replace $S$ by its complement, so we are done.
\end{proof}

\begin{corollary}
\label{GR-4.2}
The number of $(A,S)$-involutions is the same if we replace $S$ by $\{1,\ldots,k\}\backslash S$.
\end{corollary}

\begin{proof}
By Theorem \ref{GR-bijection}, the $(A,S)$-involutions are in bijection with $A$-compatible ornaments consisting only of $1$-cycles and $2$-cycles with the following conditions: (i) any $2$-cycle has vertices of distinct colors or both vertices come from the same descending block; if a $2$-cycle has exactly one vertex from a descending block, then it is not isomorphic to any other $2$-cycle; and (iii) there is at most a single $1$-cycle from each descending block.

By replacing each monochromatic $2$-cycle by two $1$-cycles of the same color, we get a bijection between the $A$-compatible ornaments described above and $A$-compatible ornaments such that (i$^\prime$) there are only $1$-cycles and $2$-cycles; (ii$^\prime$) any $2$-cycle has vertices of distinct colors; and (iii$^\prime$) if a $2$-cycle has exactly one vertex from a descending block, then it is not isomorphic to any other $2$-cycle. This latter set can also be realized as the image of the $(A,S)$-involutions under Proposition \ref{auxiliary-bijection}. Composing the bijection of the preceeding paragraph with the bijection in this paragraph gives us a bijection between $(A,S)$-involutions and $A$-compatible ornaments satisfying (i$^\prime$), (ii$^\prime$), and (iii$^\prime$).

We finally observe that if we replace $S$ by its complement, then condition (ii$^\prime$) does not change, since any cycle with exactly one descending vertex also has exactly one ascending vertex. Also, conditions (i$^\prime$) and (iii$^\prime$) do not change because they have nothing to do with whether a block is ascending or descending. Therefore, the ornaments that correspond to $(A,S)$-involutions also correspond to $(A,\{1,\ldots,k\}\backslash S)$-involutions. We can replace $S$ with $\{1,\ldots,k\}\backslash S$ in the preceeding argument to see that it is also the case that the ornaments corresponding to $(A,\{1,\ldots,k\}\backslash S)$-involutions also correspond to $(A,S)$-permutations, so we are done.
\end{proof}

\section{Conclusion and Open Problems}
\label{conclusion}


In this paper we have always considered the Gessel-Reutenauer bijection applied to $(A,S)$-permutations, but we could just as easily look at it as a map on all permutations at once or perhaps restricted to other special classes of permutations. The general map is just a map $\Phi$ that takes a permutation, writes it in cycle notation, and replaces each element in each cycle by the block it belongs to. Of course, $\Phi$ will not be injective in general.


One interesting class of permutations consists of permutations that are split into blocks of length $a_1,\ldots,a_k$, and such that the relative ordering of the permutation values within block $i$ agrees with some pre-determined permutation $\pi_i$---so $(a_1,\ldots,a_k)$-ascending permutations would be the special case when $\pi_i$ is the identity permutation for all $i$. We will call the more general case an $((a_1,\pi_1),\ldots,(a_k,\pi_k))$-permutation. For example, a $((3,312),(2,12))$-permutation would be a permutation $\pi$ such that $\pi(2) < \pi(3) < \pi(1)$ and $\pi(4) < \pi(5)$.

Unfortunately, the map $\Phi$ is not usually injective when applied to the $((a_1,\pi_1),\ldots,(a_k,\pi_k))$-permutations. For instance, let $((a_1,\pi_1),(a_2,\pi_2)) = ((3,132),(1,1))$. Then $\Phi(1423) = \Phi(2431)$, as they both yield a $1$-cycle with color $1$ and a $3$-cycle with colors $1,1,2$. If one considers large enough permutations, one can also find cases where $\Phi$ maps $3$ permutations to the same ornament. In fact, it appears that $\Phi$ fails to be injective for some set of permutations $\pi_1,\ldots,\pi_k$ whenever there is some $i$ with $a_i > 2$. (This is the smallest case that allows for a permutation $\pi_1$ that is neither always ascending or always descending.) Also, the sizes of the fibres of $\Phi$ appear to vary. Thus there is no obvious structure that is preserved, at least in terms of injectivity, when we look at $((a_1,\pi_1),\ldots,(a_k,\pi_k))$-permutations.


It also appears that Corollary \ref{EFW-symmetric} does not hold in this case. In particular, the number of $((3,123),(3,132))$-permutations in the conjugacy class $(1,2,3)$ is not the same as the number of $((3,132),(3,123))$-permutations in the conjugacy class $(1,2,3)$. (There is a single one, $\pi = 134265$ in the first case, and there are none in the second case.)

Nevertheless, one could perhaps show that when the number of inversions in each $\pi_i$ is bounded, so is the size of the pre-image of each ornament under $\Phi$, provided that we hold the number of blocks constant. One could also ask for ways to determine, for a given ornament, what the pre-image of $\Phi$ looks like. Perhaps there is some generalization of Lemma \ref{walk-lemma} that would hold. Since $\Phi$ is not injective in general, the conclusion would necessarily have to be weaker, but perhaps one could find a nice partial order or multi-dimensional sequence to use that would help to determine the relative order of two vertices in an ornament.


We could also look at the map $\Phi'$ that takes a permutation, applies $\Phi$ to obtain an $A$-compatible ornament, then obtains an $A$-compatible template from the ornament. We could ask the same sorts of questions about $\Phi'$, and will probably be more successful, since the bijection between $(A,S)$-permutations and $A$-compatible templates seems to be the most natural of all the bijections presented in this paper. The tradeoff is that the $A$-compatible templates deal less directly with cycle structure, but it should still be fairly easy to figure out the cycle structure in most cases. I believe that to understand $\Phi$, one should start by trying to understand $\Phi'$.


Another loose end is a polynomial identity from \cite{EFW}. Let $f_{\lambda}(n)$ be the generating function for permutations in $S_n$ by number of fixed points. In other words, the $\lambda^k$ coefficient in $f_{\lambda}(n)$ is the number of permutations in $S_n$ with $k$ fixed points. Eriksen et al.\ show that the number of $(a_1,\ldots,a_k)$-descending derangements is

\begin{equation}
\label{mystery-poly}
\frac{1}{a_1!\cdots a_k!} \sum_{T \subset \{1,\ldots,n\}} (-1)^{|T|} f_{\lambda}(|\{1,\ldots,n\} \backslash T|) \prod_{i=1}^k f_{\lambda}(|A_i \cap T|).
\end{equation}

This polynomial does not depend on $\lambda$, and it counts the $(a_1,\ldots,a_k)$-descending derangements. Can we generalize (\ref{mystery-poly}) to count the number of $(A,S)$-derangements? The polynomial in (\ref{mystery-poly}) looks like the sort of sum one might expect to get after applying the P\'{o}lya enumeration theorem. Can we explain why this is the case? We have made some progress towards explaining (\ref{mystery-poly}) in \cite{Ste2}, but we consider it still to be unsatisfactory.


A final direction of further inquiry involves a simpler proof of one of our results. In the proof of Corollary \ref{EFW-enumerate}, we used the fact that there are $\binom{n}{a_1,\ldots,a_k}$ $A$-good ornaments. The enumeration of $A$-good ornaments required Theorem \ref{GR-bijection}. However, the simplicity of both the question ``How many $A$-good ornaments are there?'' and its answer suggests that there should be a more direct proof that there are $\binom{n}{a_1,\ldots,a_k}$ $A$-good ornaments. Can we find such a proof? We could equivalently ask for a simple enumeration of the $A$-compatible templates.

\section{Acknowledgements}


This research was supervised by Joe Gallian at the University of Minnesota Duluth, supported by the National Science Foundation and the Department of Defense (grant number DMS 0754106) and the National Security Agency (grant number H98230-06-1-0013).

In addition to Joe Gallian, the author thanks Reid Barton, Paul Christiano, Ricky Liu, Phil Matchett Wood, and Aaron Pixton for help with the paper itself. He particularly thanks Ricky Liu and Reid Barton for pointing out the idea of templates for use in proving Theorem \ref{GR-bijection}. He also thanks Richard Peng for pointing out a correspondence between the Gessel-Reutenauer bijection and the Burrows-Wheeler transform. Finally, he thanks Geir Helleloid, Adam Hesterburg, Nathan Kaplan, Nathan Pflueger, and Yufei Zhao for helpful conversations.

\section{Glossary}
\label{defns}

This section is intended for reference only. All necessary definitions will also be given either in the introduction or the body of the paper.

\begin{itemize}

\item ascent: an index $i$ of a permutation $\pi$ on $\{1,\ldots,n\}$ such that $\pi(i) < \pi(i+1)$
\item descent: an index $i$ of a permutation $\pi$ on $\{1,\ldots,n\}$ such that $\pi(i) > \pi(i+1)$
\item $(a_1,\ldots,a_k)$-ascending permutation: a permutation that ascends in consecutive blocks of lengths $a_1,\ldots,a_k$; in other words, its descent set is contained in $\{a_1,a_1+a_2,\ldots,a_1+\cdots+a_{k-1}\}$. The blocks are referred to as $A_1,\ldots,A_k$, so, for example, $A_1 = \{1,\ldots,a_1\}$.
\item $(a_1,\ldots,a_k)$-descending permutation: a permutation that descends in consecutive blocks of lengths $a_1,\ldots,a_k$. Once again the blocks are referred to as $A_1,\ldots,A_k$.
\item $(a_1,\ldots,a_k,S)$-permutation or $(A,S)$-permutation (in this case $A = (a_1,\ldots,a_k)$ implicitly): a permutation that, when split into blocks $A_1,\ldots,A_k$ of lengths $a_1,\ldots,a_k$, descends in the blocks $A_i$ for $i \in S$ and ascends in the other blocks.
\item necklace: a directed cycle whose vertices are either colored or labelled; Figure \ref{necklace-example} has an example of a necklace. Necklaces are also sometimes called cycles.
\item fundamental period: the smallest contiguous subsequence $P$ of a necklace such that the necklace is $r$ copies of $P$ for some $r$.
\item $r$-repeating: a necklace is $r$-repeating for the value of $r$ in the preceding definition
\item ornament: a multiset of necklaces; Figure \ref{necklace-example} has an example of an ornament.
\item $(a_1,\ldots,a_k)$-compatible ornament or $A$-compatible ornament: an ornament whose vertices are colored by the integers $\{1,\ldots,k\}$ such that $a_i$ vertices are colored by $i$. In this case we are either implicitly or explicitly considering $(A,S)$-permutations as well, so there is an associated subset $S$ of $\{1,\ldots,k\}$.
\item ascending color: a color that does not lie in the set $S$ in the above definition
\item descending color: a color that does lie in the set $S$
\item $(a_1,\ldots,a_k)$-good ornament or $A$-good ornament: an $A$-compatible ornament such that every necklace is $1$-repeating
\item walk: the sequence $W(v)$ defined in Lemma \ref{walk-lemma}
\item signed walk: the sequence $A(v)$ defined in Corollary \ref{walk-corollary}
\item packet: given a vertex $v$ of an ornament, the packet of $v$ is the set $P(v) := \{v' \mid A(v') = A(v)\}$
\item successor sequence: given a sequence $X = x_0,x_1,x_2,\ldots$, the successor sequence is the sequence $S(X) = x_1,x_2,\ldots$
\item successor vertex: given a vertex $v$ in an ornament, the successor vertex $s(v)$ is the next vertex after $v$ in the relevant cycle
\item successor packet: the successor packet of a packet $P = P(v)$ is the packet $S(P) := P(s(v))$
\item orbit: given a packet $P$, the orbit $O(P)$ of $P$ is the orbit of $P$ under the map $S$, thought of as a necklace of packets; it is automatically $1$-repeating
\item template: a set of orbits
\item $A$-compatible template: a template such that every packet in a given orbit has the same size, no two orbits have the same color sequence, and there are $a_i$ vertices in total colored $i$

\end{itemize}

\nocite{*}

\bibliographystyle{plain}
\bibliography{GR}

\end{document}